\documentclass[a4paper,notitlepage,twoside,leqno,12pt]{amsart}
\usepackage{graphicx}
\usepackage{float}
\usepackage{bbm,pifont,latexsym}
\usepackage{anysize}

\marginsize{3.8cm}{3.8cm}{3.8cm}{3.6cm}

\usepackage{dcolumn,indentfirst}
\usepackage[colorlinks=false,linkcolor=black,citecolor=black,urlcolor=black,bookmarksopen=true,linktocpage=true,pdfstartview=FitH]{hyperref}
\usepackage{amsmath,amssymb,amscd,amsthm,amsfonts,mathrsfs}
\usepackage{color,graphicx,xcolor,graphics,subfigure,extarrows,caption2}
\usepackage{titletoc}
\newtheorem{thm}{Theorem}[section]

\newtheorem*{conj}{Conjecture}
\newtheorem*{main}{Main Theorem}
\newtheorem{cor}{Corollary}

\newtheorem{lema}[thm]{Lemma}

\theoremstyle{definition}
\newtheorem{defi}{Definition}
\newtheorem*{rmk}{Remark}

\usepackage{enumerate,enumitem}
\makeatletter\@addtoreset{equation}{section}\makeatother

\begin{document}

\author{Panjing Wu}
\address{Department of Mathematics, Nanjing University, Nanjing, 210093, P. R. China}
\email{pgwuxidian@163.com}

\title[]{On the  Global Curve Attractor  for polynomial gluing}

\begin{abstract}
Pilgrim's Finite Global Attractor Conjecture has been verified for  polynomials \cite{Belk2022}, but remains open for general rational maps. In this paper, we prove the conjecture for a family of rational maps obtained by  gluing two PCF polynomials along the boundaries of their finite superattracting basins.   Adapting the idea of \cite{WangZhang}, we show that a suitably defined intersection number with a finite family of separating arcs eventually decays under pullback, yielding a finite collection of homotopy classes that attracts all non-peripheral curves under iteration.

\end{abstract}

\subjclass[2010]{Primary: 37F45; Secondary: 37F10, 37F30}
\keywords{Finite global attractor, Hubbard tree, gluing, post-critically finite}
\date{\today}
\maketitle

\section{Introduction}

\subsection{The statement of the main result}
The study of post-critically finite (PCF) rational maps lies at the heart of complex dynamics. Among the many deep open problems concerning these maps, Pilgrim's \emph{Finite Global Attractor Conjecture} has stands out as a central challenge. It connects the dynamics of curves to Thurston's theory of combinatorial equivalence and to the geometry of Teichm\"uller space \cite{KPS}. The conjecture can be stated as follows \cite{P2}.

\begin{conj}
Let $f$ be a post-critically finite rational map and $P_f$ its post-critical set. Suppose $f$ is not a  Latt\`es map. Then there is a finite set of non-peripheral curves $\gamma_i, i =1, \cdots, n$, such that
 for any non-peripheral curve $\gamma$ in $\widehat{\Bbb C} \setminus P_f$, there is an $N \ge 1$ such that for any $k \ge N$,
 any non-peripheral component of $f^{-k}(\gamma)$ must be homotopic to one of these $\gamma_i$.
\end{conj}

Over the past decade, substantial progress has been made towards the conjecture. It has been verified for non-Latt\`es quadratic maps with four post-critical points \cite{KL19},  and for all critically fixed rational maps \cite{Hlu19}. Pilgrim proved that the conjecture holds when the associated virtual endomorphism on the mapping class group is contracting,   and  for   quadratic polynomials with a periodic critical point, and also for three specific quadratic polynomials \cite{P1}. Recently,  Belk,  Lanier,   Margalit  and  Winarski   proved the conjecture for all post-critically  finite polynomials. Their proof introduced a combinatorial encoding of curves via subforests of a tree, and showed that the  tree-lifting operator does not increase the metric on the complex of trees. Therefore, after sufficiently many iterations, any tree is sent into a small neighborhood of a Hubbard vertex containing only finitely many trees. Hence the pullback of any non-peripheral curve is eventually homotopic to the boundary of a Jordan neighborhood of a subforest of one of these finitely many trees. This verified  the conjecture  for polynomial case. Furthermore, Bonk,  Hlushchanka,   Iseli and  Lodge   made substantial contributions to the case of four post-critical points \cite{BHI}\cite{BHL}, which has recently been fully resolved by  Bartholdi,  Dudko and  Pilgrim \cite{BDP}. More recently, Wang and Zhang \cite{WangZhang} gave an alternative proof for the polynomial case, and their approach also yields an exponential rate of convergence to the attractor. Despite this progress, the conjecture remains open for general PCF rational maps.

The main purpose of this paper is to prove the conjecture for a class of rational maps obtained  by gluing two polynomials along the Jordan boundaries  of two finite super-attracting  immediate basins.  More precisely, suppose $f$ and $g$ are two PCF polynomials of degree $d_1$ and $d_2$ respectively, both of which have a  finite immediate super-attracting basin of degree $d_0$  such that the $\emph{independence condition}$  holds, that is, the remaining critical orbits do not intersect the marked immediate basins. According to \cite{Zhang2024}, $f$ and $g$ can be glued into a rational map  of degree $d_1+d_2-d_0$, denoted by  $\mathcal{G}(f,g)$.
 In this paper, we prove

\begin{main}
Pilgrim's Finite Global Attractor Conjecture holds for \(\mathcal{G}(f,g)\).
\end{main}

The main theorem is established under the independence condition imposed in \cite{Zhang2024}, which requires that the remaining critical orbits avoid the marked immediate basins. We remark, however, that for polynomial data this hypothesis can be dropped whenever the resulting topological gluing is post-critically finite; see the remark following Theorem~2.2. As a consequence, we obtain the following  characterization: a PCF rational map $G$ is realized as a gluing of two PCF polynomials of degrees $d_1$ and $d_2$ if and only if it admits an invariant Jordan curve $\Gamma$ such that $G|_{\Gamma}$ is conjugate to a circle covering of degree $d_0 \ge 2$, and it has two super-attracting fixed points $p$ and $q$ of local degrees $d_1$ and $d_2$, respectively, with $p$ lying in the exterior of $\Gamma$ and $q$ in its interior, and with all other preimages of $p$ contained in the interior of $\Gamma$ and all other preimages of $q$ contained in the exterior.

Our work  establishes the base case for  gluing  two polynomials. According to  \cite{Zhang2024}, one can   glue \(\mathcal{G}(f,g)\)   with another  PCF polynomial into a rational map of higher degree, provided that the required conditions are satisfied.    This suggests an inductive strategy for expanding the result to a wider class of rational maps via successive gluing with polynomials.

\begin{figure}
    \centering
    \includegraphics[width=0.63\linewidth]{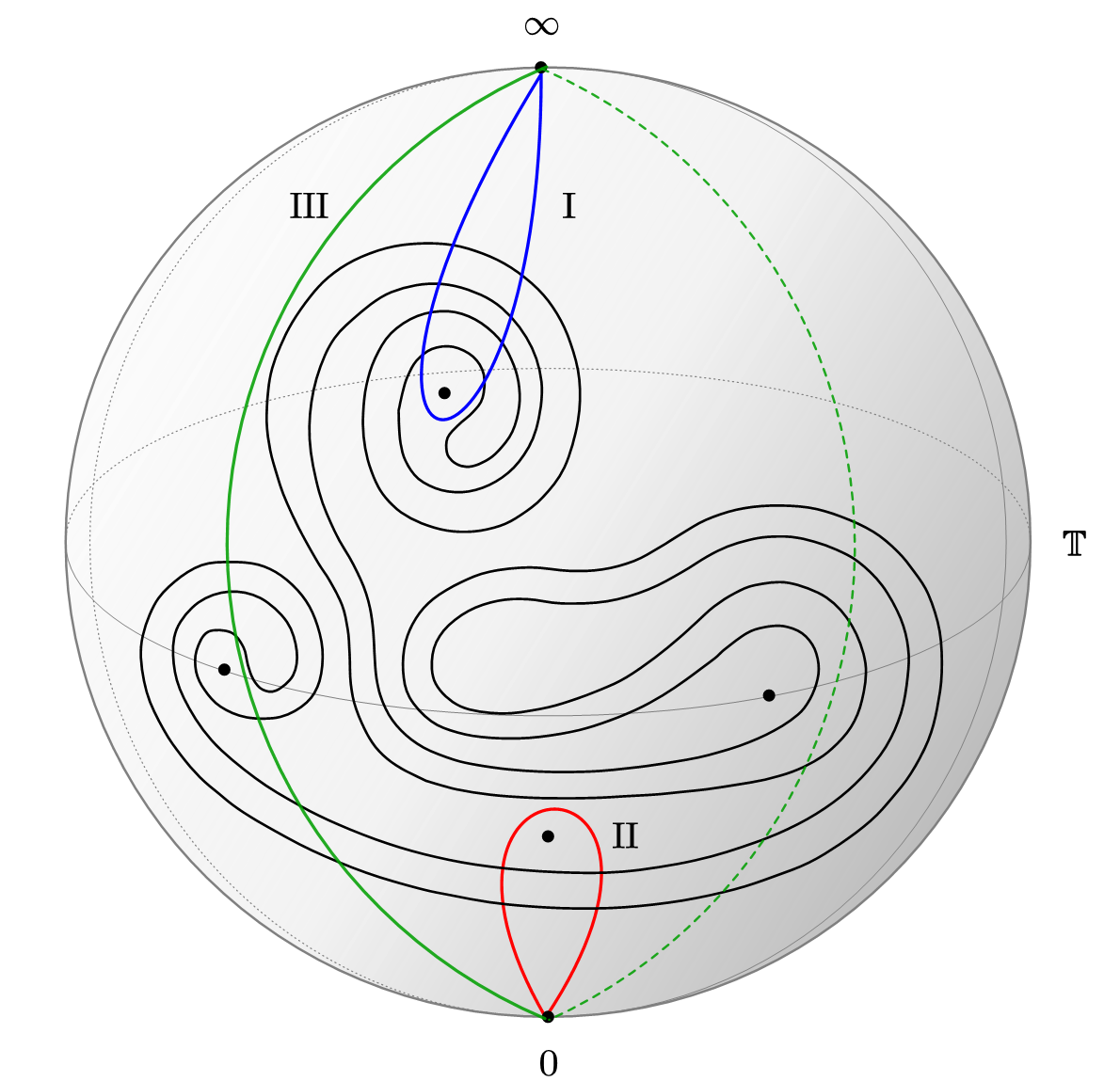}
 \caption{An admissible family of separating arcs for $F$.  Type~I arcs (blue) lie on the $f$-side, type~II arcs (red) on the $g$-side, and type~III arcs (green) cross the gluing circle $\mathbb{T}$. The family is chosen so that each complementary region contains at most one postcritical point (shown in black). These arcs are used to define a complexity function that controls the pullback dynamics.}
    \label{F5}
\end{figure}
\subsection{The idea of  the proof.}
We adapt the idea from \cite{WangZhang} for polynomials to the situation here. Since the global curve attractor is combinatorially invariant, it suffices to consider the topological gluing of $f$ and $g$, say $F$, for which the gluing curve is identified with the unit circle,  see Section~\ref{sec:gluing} for details.  First we construct an invariant graph \(T\) of $F$
by gluing the Hubbard trees of \(f\) and \(g\) along the unit circle.   We then   build a finite   family \(\mathcal{F}\) consisting of three types of separating arcs, each of which is a Jordan curve in the sphere. Roughly speaking,
 a type I separating arc is the union of two $f$-rays starting from infinity and   landing  at  the same point,
   a type II  separating arc is the union of two $g$-rays starting from the origin and  landing at the same point, a type III separating arc is the union of two $f$-rays and two $g$-rays which meet together on the unit circle.  We choose the rays such  that the family
 $\mathcal{F}$ is admissible, meaning that its complementary regions each  are simply connected and  contain at most one postcritical point.

For any non-peripheral curve \(C\subset \widehat{\mathbb{C}}\setminus P_F\), define its complexity \(N_{\mathcal{F}}(C)\) as the minimum, over all curves homotopic to \(C\), of the total intersection number with the arcs in \(\mathcal{F}\). This provides a numerical measure of complexity for such curves.  Our  main task is to prove that  there exists a constant \(M > 0\) depending only on $F$ and $\mathcal{F}$, such that for every non-peripheral curve \(C\), every component \(C'\) of \(F^{-k}(C)\) satisfies \(N_{\mathcal{F}}(C') \le M\) for all sufficiently large \(k\).
 The main theorem then follows  since the curves in
\[
\mathcal{C}_M=\bigl\{ \gamma \subset \widehat{\mathbb{C}} \setminus P_F \;\big|\; \gamma \text{ is non-peripheral and } N_{\mathcal{F}}(\gamma) \le M \bigr\}
\]
belong to  finitely many homotopy classes.

Conceptually, the global curve attractor problem is dual to the invariant graph problem in the sense that a solution to one yields a solution to the other. This idea is effectively implemented in \cite{Belk2022}, where a simple closed curve is encoded as a thin Jordan neighborhood of a subforest of a tree. It turns out that the tree lifting operator is a complex map that does not increase a certain metric on the complex of trees. Consequently, iterating this operator on any tree eventually leads to a periodic cycle of trees. By showing that any periodic tree must lie in a small neighborhood of the Hubbard tree, the global curve attractor problem for polynomial maps follows.

For general rational maps, however, a periodic graph need not lie in a small neighborhood of an invariant graph, since there exist infinitely many periodic Jordan curves passing through the post-critical set, as shown by  Meyer \cite{BM}\cite{M13}. Thus, the framework of \cite{Belk2022} does not apply to the present situation.   Nevertheless, our proof shows that as we lift a connected graph containing $P_F$ by iterating $F$, the complexity of each edge eventually becomes bounded by a constant depending only on $F$, and consequently we reach a finite neighborhood of the invariant graph $T$.

\subsection{The  organization of  the paper.} In Section~\ref{sec:gluing}, we recall the gluing construction $\mathcal{G}$ and construct the invariant graph $T$ for the topological gluing $F$. In Section~\ref{sec:arcs}, we introduce an admissible family $\mathcal{F}$ of separating arcs and define the complexity of a curve with respect to $\mathcal{F}$. We then modify $\mathcal{F}$ to a periodic family $\mathcal{F}_0$ such  that there exist  constants $s \ge 1$ and
 $c>0$ satisfying  $N_{\mathcal{F}}(\eta) \le c\, N_{\mathcal{F}_0}(C)$ for every non-peripheral curve $C$ and every component $\eta$ of $F^{-s}(C)$. In Section~\ref{sec:complexity}, we prove that the complexity of a curve with respect to $\mathcal{F}_0$ does not increase under pullback by $F$ and establish the key decay estimate. We then assemble these ingredients to complete the proof of the main theorem.

\vspace{4mm}

\noindent$\bold{Acknowledgement}$. We would like to thank  Gaofei Zhang for suggesting the problem.

\section{The Invariant Graph For the topological gluing}
\label{sec:gluing}

\subsection{Hubbard trees}

Let $f$ be a  polynomial of degree \(d\ge 2\). Recall that the \emph{filled-in Julia set} of a polynomial \(f\) is
\[
K(f)=\{z\in\widehat{\mathbb C}: f^n(z)\not\to\infty\ \text{as } n\to\infty\},
\]
which is known
to be compact and full. The postcritical set is defined as
\[
P_f=\bigcup_{n\ge 1} f^n(\Omega_f),
\]
where \(\Omega_f\) denotes the set of critical points of \(f\). The behavior of the critical points under the iteration of  the polynomial dramatically influences the topology of the set $K(f)$. We call a polynomial $f$   postcritically finite (PCF) if $P_f$ is a finite set, that is,  every critical orbit is either periodic or eventually periodic. For such polynomials, the filled-in Julia set $K(f)$ is connected and locally connected.

It is well known  that for a PCF polynomial $f$ of degree $d \ge 2$,  the set  $K(f) \setminus \{z\}$  has only  finitely many   connected components for any $z \in K(f)$. Consequently, the filled-in Julia set $K(f)$ admits a tree-like structure \cite{DH}.  To  capture this combinatorial structure more concretely, one introduces the \textit{Hubbard tree}, denoted by $H_f$. It is the unique smallest invariant tree embedded in the filled-in Julia set that contains all post-critical points of the polynomial (subject to a regularity condition within bounded Fatou components).  For a detailed treatment, we refer the reader to~\cite{DH1}\cite{DH2}\cite{Po}.

For  edges \(\tilde{e}\) and $e$ of \(H_f\) and integers $k, n \ge 1$,
 we say that \(f^k(\tilde{e})\) \emph{covers} \(e\) $n$ times, if there are $n$ disjoint subintervals  of $\tilde{e}$, say $\tilde{e}_i, 1\le i \le n$, such that $f^k: \tilde{e}_i \to e$ is a homeomorphism.

\begin{defi}
Suppose $f$ is a PCF polynomial and $e$ is an edge  of the Hubbard tree  for  $f$.  We call $e$
 \emph{expanding}   if there exists an integer $k \ge 1$ such that $f^k(e)$ covers $e$ at least twice, and  \emph{attracting} if there exists an integer $k\ge 1$ such that $f^k: e \to e$ is a homeomorphism.
\end{defi}

It is easy  to see that an expanding edge contains infinitely many repelling periodic points of $f$, whereas an attracting edge must have at least one of its two endpoints to be a super-attracting periodic point\cite{Po}. In Figure~\ref{z^2-1}, the Hubbard tree for $z^2-1$ contains a single attracting  edge for
 which both the endpoints belong to a super-attracting cycle.
\begin{figure}
    \centering
    \includegraphics[width=0.7\linewidth]{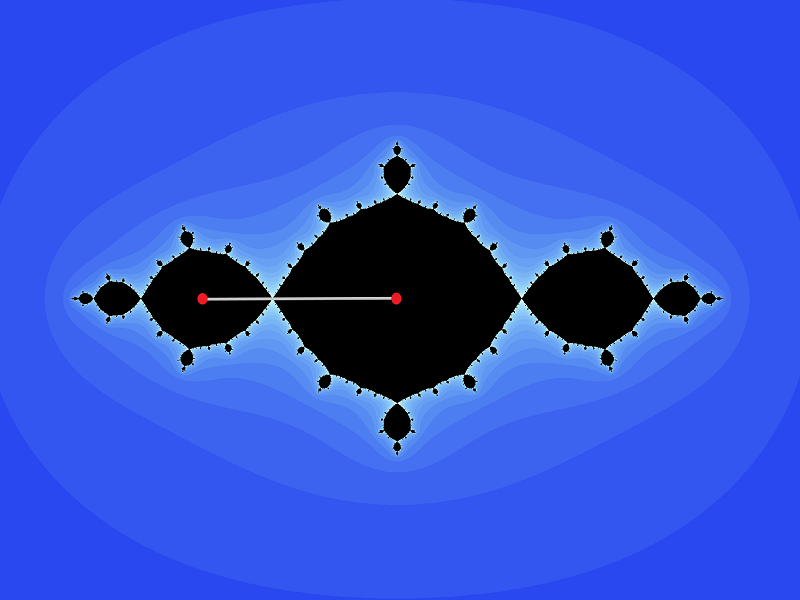}
   \caption{The Hubbard tree of $z^2-1$. The attracting edge (gray) is the union of two internal rays from $0$ and $-1$ to their common  land points.}
    \label{z^2-1}
\end{figure}

The next  lemma is crucial in the construction of the  family $\mathcal{F}$.

\begin{lema}[\cite{WangZhang}, Lemma 2.1]
\label{ea}
Let $e$ be an  edge of a Hubbard tree. Then  there exists an integer $t\ge 1$ such that $f^t(e)$ covers either an expanding edge or an attracting edge.
\end{lema}

\subsection{The gluing operator $\mathcal{G}$}

 Suppose $f$ and $g$ are two PCF rational maps of degree $d_1$ and $d_2$,  respectively. Assume  each has a marked   immediate super-attracting basin, denoted  $D_f$ and $D_g$, whose boundaries are Jordan curves   and which have    the same degree $d_0$. We further assume the \textit{independence condition} that the other critical orbits of $f$ (resp. $g$) are disjoint from  $D_f$ (resp. $D_g$). Let  $\mathcal{R}_{d, d_0}$   denote the space of rational maps  satisfying  the above properties.

Now we can construct a topological map $F: S^2 \to S^2$ by gluing $f$ and $g$ along  $\partial D_f$ and $\partial D_g$.  More precisely, let $\mathbb{D}$ be the unit disk and let $\phi: D_f \to \mathbb{D}$, $\psi: D_g \to \mathbb{D}$ be the holomorphic isomorphisms that  conjugate $f$ and $g$ to the power map $z \to z^{d_0}$. For each $1 \le k \le d_0-1$, let $R_k$ denote the rigid rotation given by the angle $2 k \pi i/(d_0-1)$, that is,
$$
R_k: z \mapsto e^{2 k \pi i/(d_0-1)} z.
$$
Let $S$ denote the conjugation map, that is, $S(z) = \bar{z}$. Then
there is a homeomorphism
\[\label{m} \Phi = \phi^{-1} \circ R_k \circ S \circ \psi: \partial D_g \to \partial D_f,\]  which reverses the orientation.  We   extend $\Phi$ to a homeomorphism of the sphere so that it maps $\overline { D_g}$ to $ D_f^c$ and maps $D_g^c$ to $\overline {D_f}$.  We then   glue the complements  $ D_f^c$ and $D_g^c$ along their boundaries by identifying $x \in \partial D_g^c$  with  $\Phi(x) \in \partial D_f^c$.  The topological space obtained in this way,
$$
X = D_f^c \bigsqcup_{x \sim \Phi(x)} D_g^c,
$$  is homeomophic to $S^2$.   Now we  define a topological map
$$
F: X \to X
$$
by setting
\[\label{oo}
F(z) =
\begin{cases}
f(z) & \text{ for $z \in D_f^c$ and $f(z) \in D_f^c$}, \\
\Phi^{-1}\circ f(z) & \text{ for $z \in D_f^c$ and $f(z) \in D_f$}, \\
g(z) & \text{ for $z \in D_g^c$ and $g(z) \in D_g^c$}, \\
\Phi\circ g(z) & \text{ for $z \in D_g^c$ and $g(z) \in D_g$}.
\end{cases}
\]

\noindent We call $F$  the $\emph{topological gluing}$ of $f$ and $g$.    In the case that $F$ is post-critically finite and has no Thurston obstructions, by Thurston's characterization theorem \cite{DH},
there is a rational map $G$ which is combinatorially equivalent to $F$.  In this case, we  say $G$ is the $\emph{geometric gluing}$ of $f$ and $g$   and denote
$$
G = \mathcal{G}(f, g).
$$

\begin{thm}[Zhang, \cite{Zhang2024}]\label{THZ}
Let $d_0 \ge 2$.  Suppose $f  \in \mathcal{R}_{d_1,d_0}$  and $g  \in \mathcal{R}_{d_2,d_0}$ such that at least one of them is a polynomial.
Then $f$ and $g$  can be always glued  into a rational map $R$ of degree $d_1 + d_2 -d_0$. In particular, two polynomials can be always glued into a rational map.
\end{thm}

\begin{rmk}
We observe that, when considering only the gluing of two polynomials, the independence condition can be dispensed with. More precisely, we may allow the critical points of $f$ and $g$ to be captured by the marked immediate basins, provided that their forward orbits are finite.  To prove the absence of obstructions in the topological gluing, we replace the use of the number of maximal arcs employed in \cite{Zhang2024} with the intersection number of the curves in the obstruction with a family of separating arcs. This family is analogous to the collection $\mathcal{F}$ constructed in the next section, with the property that each arc in $\mathcal{F}$ is mapped homeomorphically onto some periodic arc in $\mathcal{F}_0$.

The detailed reasoning is as follows. Suppose the gluing of $f$ and $g$ admits an obstruction, and let $\Gamma$ be an irreducible one. Then every element of $\Gamma$ intersects some periodic arc in $\mathcal{F}_0$. We may further assume that the intersection number of each element of $\Gamma$ with the periodic arcs in $\mathcal{F}_0$ is minimal. Under this minimality assumption, we show that each $\gamma \in \Gamma$ lies in a L\'evy cycle. Indeed, if for some $n \ge 1$, the set of components of $f^{-n}(\gamma)$ contains a component $\gamma'$ homotopic to $\gamma$ and another component $\eta'$ homotopic to some $\eta \in \Gamma$, then there exists a periodic arc $S \in \mathcal{F}_0$ that intersects both $\gamma'$ and $\eta'$.  The sets  $\gamma' \cap S$ and $\eta' \cap S$  are disjoint, and each is mapped homeomorphically onto a distinct subset of $\gamma \cap S$.  This contradicts the minimality of $\gamma$. Hence, $\gamma$ must belong to a L\'evy cycle. However, this contradicts the fact that the inverse branch of the topological gluing contracts the orbifold metric.
With a slight modification of the proof presented in this paper, the main theorem also holds for such rational maps (as communicated by G. Zhang).
\end{rmk}

\subsection{Topological model and invariant graph}

Suppose  $f \in \mathcal{R}_{d_1,d_0}$ and $g \in \mathcal{R}_{d_2,d_0}$ are both polynomials. Without loss of generality, we identify the unit disk $\mathbb{D}$ with the marked immediate superattracting basin for both maps, and assume that $f$ and $g$ act as  $z \mapsto z^{d_0}$ on $\mathbb{D}$. Denote by $H_f$ and $H_g$ the Hubbard trees of $f$ and $g$, respectively, and set
\[
H_f^{\mathrm{ext}} = H_f \cap \mathbb{D}^c,\quad
H_g^{\mathrm{ext}} = H_g \cap \mathbb{D}^c,
\]
as shown in Figures~\ref{f01} and~\ref{f02}.

As in  Section~2.2, we construct  the topological model $F$ by gluing two copies of  $\mathbb{D}^c$ along the unit circle $\Bbb T$ through the orientation-reversing homeomorphism
\[
\iota(\zeta)=\frac{e^{2\pi i k/(d_0-1)}}{\zeta},\quad 1\le k\le d_0-1,
\]
for a fixed integer $k$.  The topological gluing $F : S^2 \to S^2$ is then given by
\[\label{oo}
F(z) =
\begin{cases}
f(z) & \text{ for $z \in \mathbb{D}^c$}, \\
\iota^{-1}\circ g\circ \iota(z) & \text{ for $z \in \mathbb{D}$},
\end{cases}
\]
which acts as $z\mapsto z^{d_0}$ on the gluing circle $\mathbb{T}$. Note that $ \iota = \iota^{-1}$.

 Define the   graph
\[
T = H_f^{\mathrm{ext}} \;\cup\; \iota(H_g^{\mathrm{ext}}) \;\cup\; {\mathbb{T}},
\]
with   $$V(T) = (V(H_f) - \Bbb D) \bigcup \iota(V(H_g) - \Bbb D))$$  where $V(\cdot)$ denotes the vertex set of a graph. We call each component of $T - V(T)$ an edge of $T$.

For brevity, we set  $T^{\mathrm{ext}}=H_f^{\mathrm{ext}}$ and $T^{\mathrm{int}}=\iota(H_g^{\mathrm{ext}})$.  We thus have
\[
T = T^{\mathrm{ext}} \cup T^{\mathrm{int}} \cup {\mathbb{T}},
\]
see Figure~\ref{f}. By construction, $T$ is a finite connected graph containing $P_F$ and satisfying $$F(T) \subset T,$$ which follows from  the independence condition on the critical orbits of $f$ and $g$.   Since $G = \mathcal{G}(f,g)$ is combinatorially equivalent to  $F$,  it suffices to prove the main theorem for $F$.
\begin{figure}[H]
\centering
\begin{minipage}[t]{0.43\textwidth}
\includegraphics[width=1.1\textwidth]{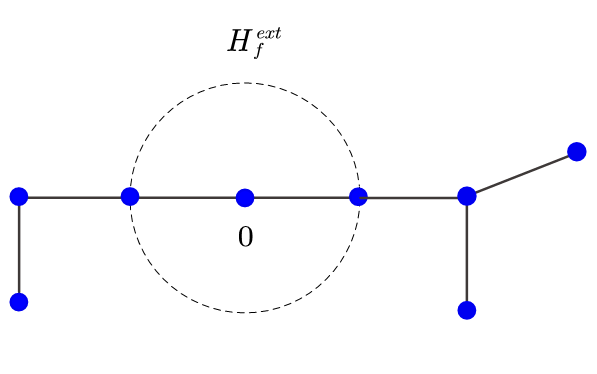}
\caption{The part $H_f^{\mathrm{ext}}$ of $H_f$ outside $\mathbb{D}$.}
\label{f01}
\end{minipage}
\hfill
\begin{minipage}[t]{0.43\textwidth}
\includegraphics[width=0.82\textwidth]{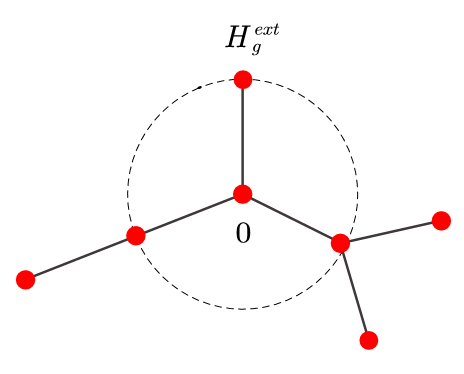}
\caption{The part $H_g^{\mathrm{ext}}$ of $H_g$ outside $\mathbb{D}$.}
\label{f02}
\end{minipage}
\end{figure}
\begin{figure}[!htpb]
\centering
\includegraphics[width=84mm]{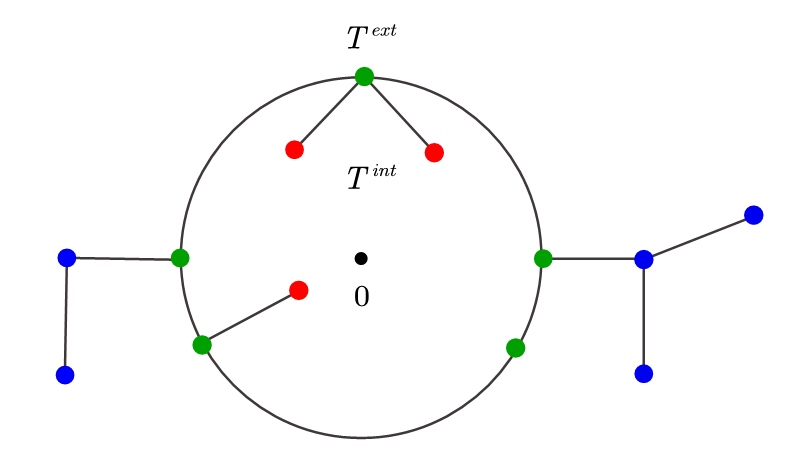}
\caption{The invariant graph $T$ gluing $H_f$ and $H_g$ along $\mathbb{T}$. }
\label{f}
\end{figure}

\section{Separating Arcs for $F$}
\label{sec:arcs}

In this section, we construct a finite family of separating arcs for the topological map $F$. These arcs will be used to define a numerical complexity for non-peripheral curves, which is the main tool in our proof.

\subsection{Complexity with respect to $\mathcal{F}$}
We call  $L \subset \widehat{\Bbb C}$  a $\emph{separating arc}$ if $L$ is a piecewise smooth Jordan curve.

\begin{defi} A family  $\mathcal{F}$   of separating arcs is called  \emph{admissible}
  if \begin{itemize}  \item the union $\bigcup_{L\in\mathcal{F}} L$ is connected, and \item each component of $\widehat{\mathbb{C}}\setminus\bigcup_{L\in\mathcal{F}}L$ contains at most one post-critical point of $P_F$. \end{itemize}
\end{defi}

Let \(\mathcal{F}\) be an admissible family of separating arcs. Let $N(\cdot, \cdot)$  denote the  intersection number of two curves.   Assume that  $C\subset \widehat{\mathbb{C}} \setminus P_F$ is a
 non-peripheral curve. We define the \textit{complexity} of $C$ with respect to   $\mathcal{F}$  as
 \begin{equation}\label{cpx}
N_{\mathcal{F}}(C) \;=\; \min_{\gamma\sim C}\; \sum_{L\in\mathcal{F}} N(\gamma,L),
\end{equation}
where the minimum is taken over all simple closed curves $\gamma$ homotopic to $C$ in $\widehat{\mathbb{C}}\setminus P_F$.

Let us explain why this complexity captures essential information about the homotopy class of \(C\). Take a minimal representative \(\gamma\) of a non-peripheral curve \(C\) that attains \(N_{\mathcal{F}}(C)\). Since \(\bigcup_{L\in\mathcal{F}}L\) is connected, each component of \(\widehat{\mathbb{C}}\setminus\bigcup_{L\in\mathcal{F}}L\) is simply connected, and by admissibility each such component has piecewise smooth boundary and contains at most one point of \(P_F\).  Note that the intersection points  cut \(\gamma\) into \(N_{\mathcal{F}}(C)\) subarcs, each contained in a single component of \(\widehat{\mathbb{C}}\setminus\bigcup_{L\in\mathcal{F}}L\).  Hence there exists a constant \(\delta>0\) depending only on \(\mathcal{F}\) such that the length of each  subarc with respect to the spherical metric is bounded above by \(\delta\) up to homotopy. Consequently,
\[
\ell_\rho(\gamma) \le \delta \cdot N_{\mathcal{F}}(C),
\]
with \(\rho\) being the spherical metric on \(\widehat{\mathbb{C}}\). It follows that an upper bound on \(N_{\mathcal{F}}(C)\) forces a uniform bound on the spherical length of \(C\). Moreover, since \(P_F\) is finite, only finitely many homotopy classes of curves can have bounded length. Thus it suffices to show that
the  complexity of the pull back, with respect to $\mathcal{F}$,  will be eventually bounded by some uniform constant.

\subsection{Constructing   an  admissible   family \(\mathcal{F}\)} We will construct three types of separating arcs $L$ according to whether $L$ belongs to the $f$-side, $g$-side or crosses the unit circle.

Let us first describe the construction of separating arcs of type I, which are composed of external rays (and internal rays, when an attracting basin is involved) of the polynomial $f$.
For each $\tilde{e} \subset T^{\mathrm{ext}}$, by Lemma~\ref{ea} there exists    an integer $t_{\tilde{e}} \ge 1$ such that $F^{t_{\tilde{e}} }(\tilde{e})$ covers either an expanding edge or an attracting edge, denoted $e$.

  If $e$ is expanding, we can select a bi-accessible repelling periodic point $x$ in the interior of $e$, say of period $p$, and an open subinterval $I \subset e$ containing $x$ such that $F^p: I \to e$ is an orientation-preserving homeomorphism.   Then choose a preimage $\tilde{x} \in \tilde{e}$ of $x$ under $F^{t_{\tilde{e}}}$  such that  $\tilde{x}$ lies in the interior of $\tilde{e}$ and is bi-accessible.
 Let $R_1^\infty$ and $R_2^\infty$ be the two external rays of $f$  landing at $\tilde{x}$, denote
\[
L_{\tilde{e}} =R_{1}^\infty \cup\{\tilde{x}\}\cup R_{2}^\infty.
\]
It follows that $L_{\tilde{e}}$  is a separating arc,  and its image $F^{t_{\tilde{e}}} (L_{\tilde{e}})$  is periodic.

\begin{figure}[H]
    \centering
    \includegraphics[width=1.02\linewidth]{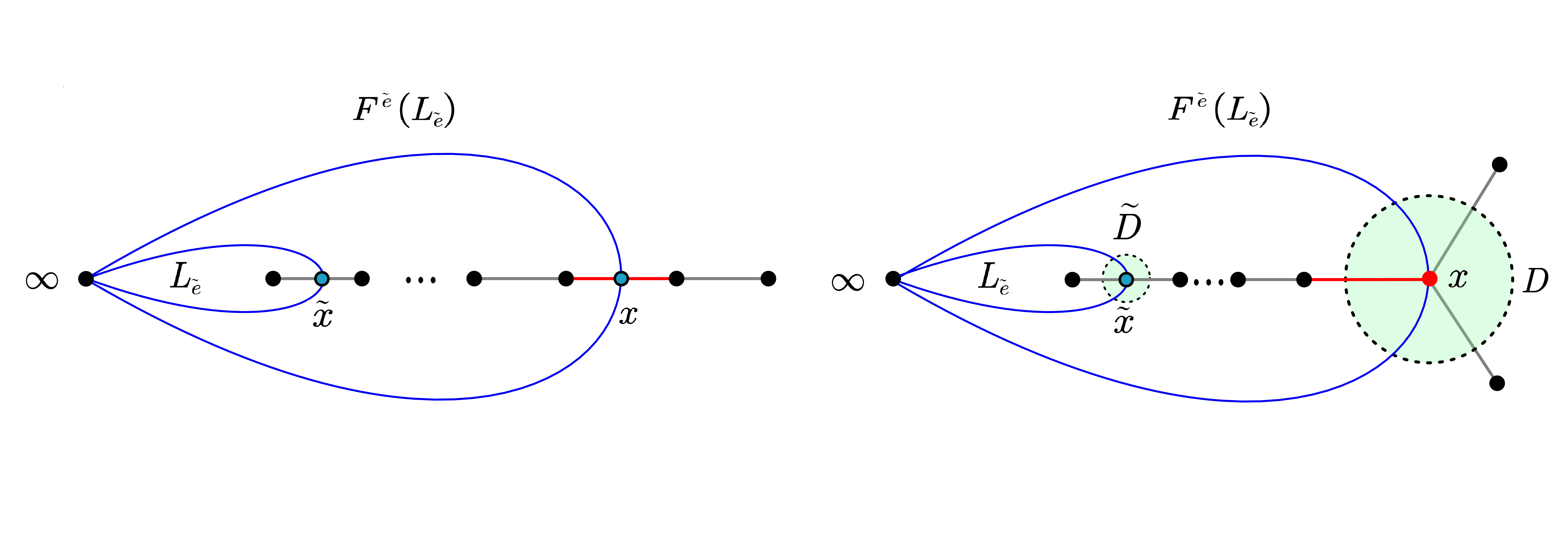}
     \caption{Two cases of separating arcs of type~I are illustrated - expanding case and attracting case. }
\label{fig:placeholder}
\end{figure}

Suppose instead that $e$ is attracting.
Let  $x$ be the super-attracting periodic point of $e$ (necessarily an endpoint of that attracting edge) and choose $\tilde{x}\in \tilde{e}$ such that $F^{t_{\tilde{e}}}(\tilde{x})=x$. Let \(D\) denote the immediate basin of \(x\).  Now choose  two periodic points   $x_1, x_2 \in \partial D\setminus P_F$ such that at each of them,  exactly one external ray lands. Let $\tilde{D}$ be the pre-image of $D$ containing $\tilde{x}$ and $\tilde{x}_1, \tilde{x}_2 \in \partial \tilde{D}$ such that $F^{t_{\tilde{e}}}(\tilde{x}_i) = x_i$ for $i = 1, 2$.
 Let $R_1^\infty$ and $R_1^{\tilde{D}}$ denote the external ray and (pre)internal ray landing at  $\tilde{x}_1$, $R_2^\infty$ and $R_2^{\tilde{D}}$   the external ray and (pre)internal ray landing at $\tilde{x}_2$,  respectively.
 Define
\[
L_{\tilde{e}}  = R_1^\infty \cup \{\tilde{x}_1\} \cup R_1^{\tilde{D}}  \cup \{\tilde{x}\} \cup R_2^{\tilde{D}} \cup \{\tilde{x}_2\} \cup R_2^\infty.
\]
 It follows that $L_{\tilde{e}}$ is a separating arc and its image $F^{t_{\tilde{e}}} (L_{\tilde{e}})$   is periodic.

The construction of
separating arcs of type II  is completely the same as that of type I,   by considering \(\tilde{e}\)  in  \(T^{\mathrm{int}}\) and  with \(f\) replaced by \(\tilde{g} = \iota^{-1} \circ g \circ \iota. \)

%If $e\subset T^{\mathrm{int}}$, the construction is analogous, with all rays taken as internal rays -- there exist two distinct (pre)internal rays from $x$ and two distinct internal rays landing at their endpoints, yielding a type~II arc $L_e$ with both endpoints at $0$, and $F^{t_e}(L_e)$ consists of two periodic internal rays from $c_0$ yielding a type~II arc $L_{e'}$.

To construct a separating arc of type III, we first choose two distinct  periodic points, say ${a}, {b} \in \Bbb T \setminus P_F$  such that at  each of them, exactly one $f$-ray and one  $\tilde{g}$-ray land.  Since  the dynamics of $F|_{\mathbb{T}}$ is the expanding covering map $z\mapsto z^{d_0}$ with $d_0 \ge 2$, such points are dense in $\Bbb T$.  Let  $R_1^\infty$, $R_1^0$ be the  $f$-ray and
$\tilde{g}$-ray  landing at ${a}$, and $R_2^\infty$,  $R_2^0$ be the $f$-ray and  $\tilde{g}$-ray   landing at  ${b}$. Then
\[
L = R_1^\infty \cup \{{a}\} \cup R_1^0 \cup R_2^0\cup \{{b} \} \cup R_2^\infty
\]
 is a separating  arc containing $\infty$ and $0$ and intersecting $\mathbb{T}$ exactly at ${a}, {b}$. Since both ${a}$ and ${b}$ are periodic,   it follows that $L$ is   periodic.

To construct an admissible family $\mathcal{F}$ consisting of the above three types separating arcs, we need only choose a type I  arc for each \(\tilde{e}\)  in  \(T^{\mathrm{ext}}\)  and a type II arc for each  \(\tilde{e}\)  in  \(T^{\mathrm{int}}\) and sufficiently many type III arcs so that the post-critical points in $\Bbb T$  are well separated. Henceforth, we fix such an admissible family and denote it by $\mathcal{F}$.

\subsection{The periodic family \(\mathcal{F}_0\) and its adjoint arcs}

For each arc $L \in \mathcal{F}$, let $t_L\ge1$ be the integer associated to its corresponding separating arc as given by above section such that $F^{t_L}(L)$ is a periodic arc. Define
\[
\mathcal{F}_0 = \bigcup_{L \in \mathcal{F}} \mathcal{O}\bigl(F^{t_L}(L)\bigr),
\]
where $\mathcal{O}(F^{t_L}(L))$ denotes the forward orbit of $F^{t_L}(L)$ under $F$.  Then $\mathcal{F}_0$ is a finite family of periodic arcs. Let $p$ be the common period of these arcs such that  $F^p(L)=L$ for every $L\in\mathcal{F}_0$.

In general, the family $\mathcal{F}_0$ does not satisfy the separation property required of an admissible family. Nevertheless, we can define the complexity with respect to $\mathcal{F}_0$ in the same manner as in (\ref{cpx}). The following lemma shows that the complexity defined using the periodic family $\mathcal{F}_0$ dominates that defined using the admissible family $\mathcal{F}$, up to a multiplicative constant that depends only on $\mathcal{F}$ and $\mathcal{F}_0$, which in turn depend only on $F$.

\begin{lema}
\label{lem:eq}
There exist constants $s\ge 1$ and $c>0$, depending only on $\mathcal{F}$ and $\mathcal{F}_0$, such that for every non-peripheral curve $C$ and every component $\eta$ of $F^{-s}(C)$,
\[
N_{\mathcal{F}}(\eta) \le c\, N_{\mathcal{F}_0}(C).
\]
\end{lema}

\begin{proof}
For each $L\in\mathcal{F}$, let $t_L\ge1$ be the integer in the construction of $\mathcal{F}_0$  such that $F^{t_L}(L)\in\mathcal{F}_0$. Take $s = \max_{L\in\mathcal{F}} \{t_L\}$. Since $\mathcal{F}_0$ is forward invariant under $F$, we have $F^s(L)\in\mathcal{F}_0$ for every $L\in\mathcal{F}$.

Take a simple closed  curve $\gamma$  in the homotopy class of $C$ such that  $$\sum_{\tilde{L} \in\mathcal{F}_0}N(\gamma, \tilde{L} )=N_{\mathcal{F}_0}(C).$$ For any component $\eta$ of $F^{-s}(C)$, choose a simple closed curve $\eta'$ homotopic to $\eta$ such that $F^s(\eta') = \gamma$.   Since $F^s$ maps each $L\in\mathcal{F}$ homeomorphically onto its image in $\mathcal{F}_0$, we have
\[
 N(\eta', L) \le   N(\gamma, F^s(L)) \le \, N_{\mathcal{F}_0}(C), \:\:\forall L \in \mathcal{F}.
\]
Thus
 $$N_{\mathcal{F}}(\eta)=N_{\mathcal{F}}(\eta') \le |\mathcal{F}|\, N_{\mathcal{F}_0}(C).$$ Set $c=|\mathcal{F}|,$ then we complete the proof.
\end{proof}

In view of this estimate, we restrict our attention to the periodic family \(\mathcal{F}_0\) and its associated complexity in what follows.

To analyze the behavior of complexity under pullback, we examine the periodic structure of $\mathcal{F}_0$ and introduce the notion of adjoint arcs. More precisely, it is necessary to choose  an integer $r \ge 1$ sufficiently large so that for every $L \in \mathcal{F}_0$ with an orientation given,  there exists a separating arc $L'$ close to $L$ and  with the induced orientation,  such that    $F^{rp}$ maps $L'$ homeomorphically onto $L$ and preserves  the orientation.  Recall that $p$ is the common period of all $L \in \mathcal{F}_0$.  The construction of such $L'$ depends on whether $L$ is of type~I, type~II, or of type~III, as illustrated in Figures~\ref{F6}-~\ref{F7}.

\begin{figure}[H]
    \centering
    \includegraphics[width=0.45\linewidth]{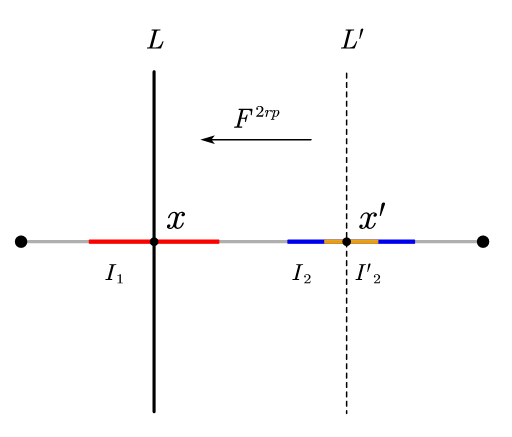}
    \caption{Constructing an adjoint arc \(L'\) for a type~I arc \(L\) - the first case.}
    \label{F6}
\end{figure}

For type~I  arcs, the construction of an adjoint arc \(L'\) splits into two cases depending on the dynamics of the corresponding periodic edge.  In the first case, suppose the periodic edge \(e\) associated with \(L\in \mathcal{F}_0\) is expanding. As constructed in Section 3.2, a type I arc \(L\)  consists of two periodic external rays \(R_1^\infty, R_2^\infty\) landing at a repelling periodic point \(x \in \operatorname{int}(e)\), and \(L\) separates \(e\).  By the choice of $x$, there is an open interval neighborhood $I$ of $x$  in $e$ such that $F^p: I \to e$ is an orientation-preserving homeomorphism. Since \(e\) is expanding, there exists an integer \(r \ge 1\) and disjoint subintervals  $I_1$ and $I_2$ of $I$  such that the restrictions
\[
F^{rp}|_{I_1} : I_1 \rightarrow e,\quad F^{rp}|_{I_2} : I_2 \rightarrow e
\]
are homeomorphisms with \(x \in \operatorname{int}(I_1)\). One can select a subinterval \( I'_{2} \subset I_2\) such that  \(F^{rp}(I'_{2}) = I_2\).  If \(F^{rp}|_{I_2}\) reverses orientation, replace \(r\) by \(2r\) and work with \(I'_{2}\) instead. Otherwise, we retain $r$ and work with $I_2$.  After this possible adjustment, we obtain an interval, either $I_2$ or $I'_{2}$, on which $F^{rp}$ is an orientation-preserving homeomorphism onto $e$.

One therefore obtains a point $x'$ in $I_2$ (or $I_2'$, depending on the choice above) such that $F^{rp}(x')=x$. In particular, $x' \ne x$.  We then lift the two external rays landing at $x$ and obtain rays $R'_1$ and $R'_2$ landing at $x'$ on the same sides of $e$ as the original rays $R_1$ and $R_2$ land at $x$.  Define \(L'\) as their union. It follows that  $$F^{rp}: L' \to L$$ is a homeomorphism with the orientation being preserved.  It is clear that the thin region  bounded by \(L\) and \(L'\) contains no post-critical point.

\begin{figure}[H]
    \centering
    \includegraphics[width=0.5\linewidth]{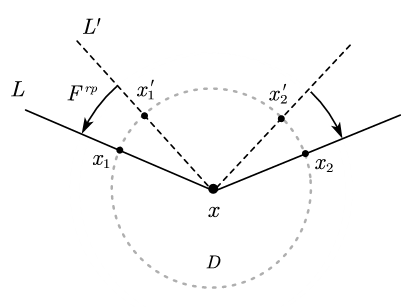}
 \caption{Constructing    an adjoint arc \(L'\) for a type~I   arc \(L\) - the second case.}
    \label{F7}
\end{figure}

In the second case, suppose the periodic edge \(e\) associated to \(L\in \mathcal{F}_0\) is  attracting. Such an arc consists of two periodic internal rays emanating from a super-attracting periodic point \(x\) and two external rays extending them.  Let \(D\) be the immediate basin of \(x\) and  \(\{x_1, x_2\} = L \cap \partial D\). Then $x_1, x_2 \notin P_F$ by the construction in Section 3.2.
 Since \(F^p(L)=L\) and the restriction \(F^p|_{\partial D}:\partial D\to\partial D\) is a covering map of degree at least \(2\),  there exists a sufficiently large integer \(r \ge 1\) and points \(x_1', x_2' \in \partial D\) on the same side of \(L\) such that
\[
F^{rp}(x_i') = x_i\quad\text{for } i=1,2.
\]  Define \(L'\) to be the union of the two internal rays from \(x\) to \(x_1'\) and \(x_2'\) together with the two external rays that continue them. By selecting \(x_1',x_2'\) sufficiently close to \(x_1,x_2\) along \(\partial D\), the two   thin regions bounded by \(L\) and \(L'\)   contain  no   post-critical point. In particular, a given
 orientation of $L$  induces an orientation of $L'$ through the deformation in the thin region bounded by $L$ and $L'$. Then
 $$
 F^{rp}:L' \to L
 $$ is an orientation-preserving homeomorphism.

Similarly, the above construction for both cases applies equally to type~II arcs.

\begin{figure}
    \centering
    \includegraphics[width=0.6\linewidth]{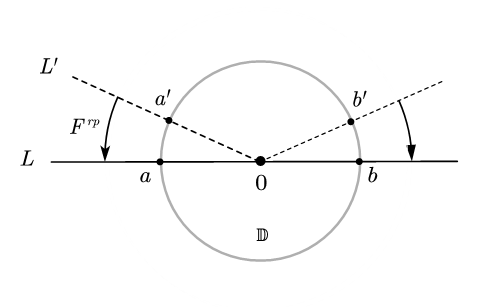}
 \caption{Construction of an adjoint arc \(L'\) for the type~III    arc \(L\).}
    \label{F7}
\end{figure}

For type III arcs that cross the gluing circle \(\mathbb{T}\), the construction is analogous to the second case for Type I (or II) arcs.  Let \(L \in
 \mathcal{F}_0\) and  \(a, b\)  be the  intersection points of \(L\) with \(\mathbb{T}\).  Since \(F^{p}|_{\mathbb{T}} : \mathbb{T} \to \mathbb{T}\) is the power map of degree \(d_0^{p} \ge 2\), we can choose an integer \(r \ge 1\) sufficiently large  and   \(a', b' \in \mathbb{T}\)
   in the same side of $L$  such that    $$F^{rp}(a') = a \hbox{,\: } F^{rp}(b') = b. $$
  Then we  take the unique lift of $L$ under \(F^{rp}\) that intersects $\Bbb T$ at $a'$ and $b'$.   Thus, \(L'\) is a separating  arc of type~III such that $F^{rp}:L' \to L
 $ is a homeomorphism which preserves the orientation. Similarly, by taking $a'$ and $b'$ sufficiently close to $a$ and $b$ respectively, we can assure that the thin regions  bounded by  \(L\) and \(L'\)  contain
   no post-critical point.

\begin{figure}[H]
    \centering    \includegraphics[width=1\linewidth]{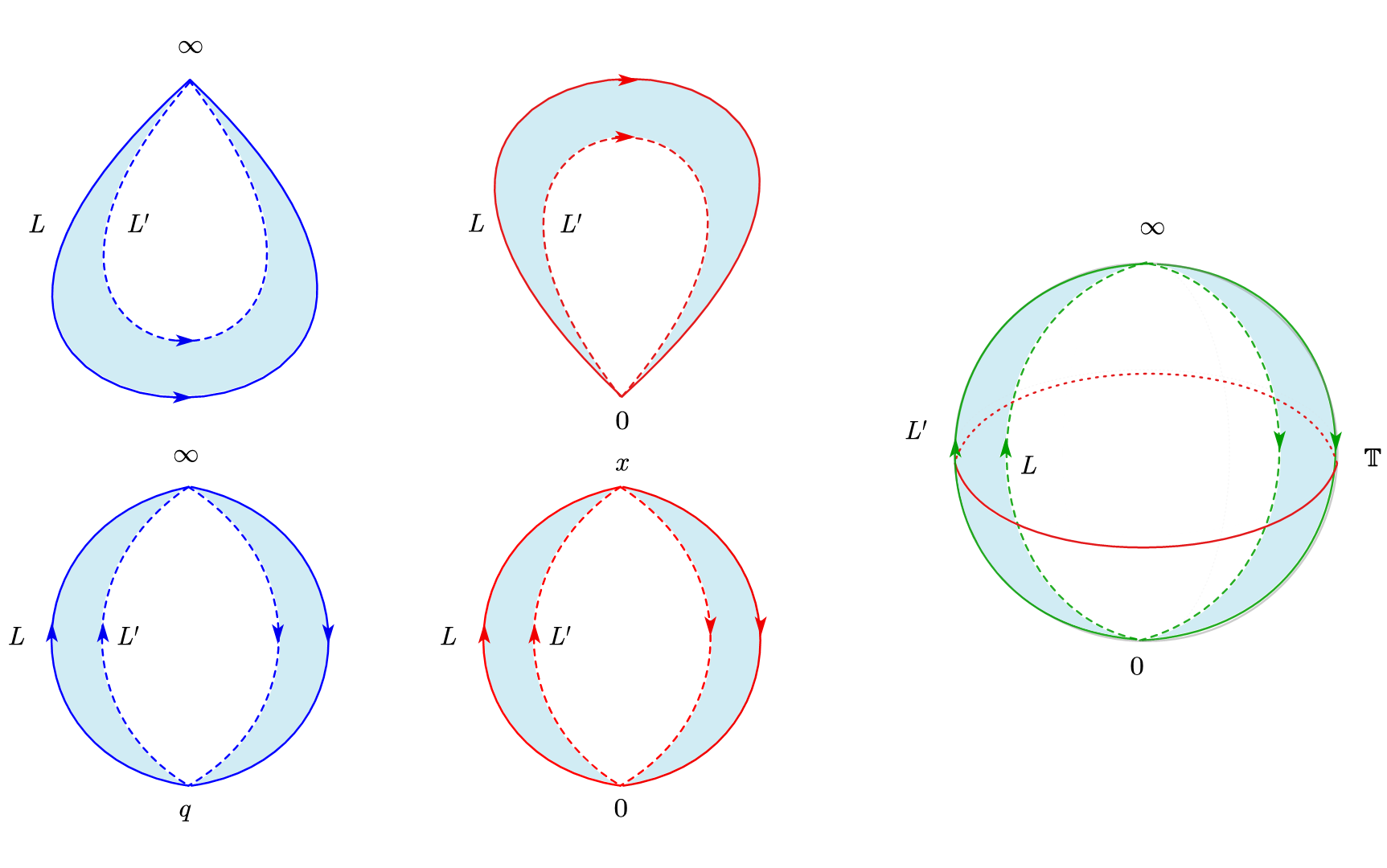}
    \caption{ The adjoints for three types of separating arcs. }
    %On the left, both arcs are of type~I and lie entirely in the basin of \(\infty\). In the center, both arcs are of type~II and are contained completely in the basin of \(0\). On the right, both arcs are of type~III and intersect the gluing curve \(\mathbb{T}\) at distinct points \(x\) and \(x'\). In all cases, the shaded region represents the strip \(S\), whose interior contains no post-critical points.}
    \label{F3}
\end{figure}

Therefore, for every \(L \in \mathcal{F}_0\) we have constructed a separating  arc  $L'$  nearby $L$ such that \(F^{rp}(L') = L\) and the restriction \(F^{rp}|_{L'}\) is orientation-preserving. Moreover, the thin region(s) bounded by \(L\) and \(L'\) contains no  post-critical point. By taking $L'$ sufficiently close to $L$, we can ensure that all the thin regions bounded by $L$ and $L'$ are disjoint with each other.     We call such \(L'\) an \emph{adjoint} of \(L\).

Clearly, by replacing $r$ with a   common multiple of the integers obtained above, we may assume that there exists an $r \ge 1$ such that for every $L \in \mathcal{F}_0$, the map $F^{rp}: L' \to L$ is an orientation-preserving  homeomorphism.

\section{Curve Complexity and Its Properties}
\label{sec:complexity}

In this section, we prove that the complexity of $\mathcal{F}_0$ does not increase under pullback, and that after a bounded number of iterates it must eventually drop.

\begin{lema}
\label{lem:mo}
For any non-peripheral curve $C$ and any component $C'$ of $F^{-1}(C)$,
\[
{N}_{\mathcal{F}_0}(C') \le {N}_{\mathcal{F}_0}(C).
\]
\end{lema}
\begin{proof}
Let \(C\) be a minimal representative that achieves \(N_{\mathcal{F}_0}(C)\). For any component \(C'\) of \(F^{-1}(C)\) and any \(L\in\mathcal{F}_0\), since the restriction \(F|_L : L \to F(L)\) is a homeomorphism, we have \(N(C',L) \le N(C, F(L))\). Given that \(F : \mathcal{F}_0 \to \mathcal{F}_0\) is bijective,  it follows that
\[
N_{\mathcal{F}_0}(C')  \le \sum_{L\in\mathcal{F}_0} N(C', L)  \le \sum_{L\in\mathcal{F}_0} N(C, F(L))  = N_{\mathcal{F}_0}(C).
\]
This completes the proof.
\end{proof}

We can now establish the key decay estimate. It asserts that if the complexity of a curve exceeds a certain threshold, then after a bounded number of pullbacks,  every component must have strictly smaller complexity.

\begin{lema}
\label{prop:decay}
There exist constants $M>0$ and $m\ge 1$  depending only on $F$ such that for any non-peripheral curve $C$ with $N_{\mathcal{F}_0}(C)>M$, every  component $C'$ of $F^{-m}(C)$ satisfies  $N_{\mathcal{F}_0}(C') < N_{\mathcal{F}_0}(C)$.
\end{lema}

\begin{proof}
Let $M =   d^{rp}$  and $m = rp$ with $d$ being the degree of $F$, $p$  the common period of the separating arcs in $\mathcal{F}_0$,
 and $r \ge 1$  the integer in the construction of adjoint arcs.
 Suppose  that $C$ is a minimal representative realizing $N_{\mathcal{F}_0}(C)$ such  that $N_{\mathcal{F}_0}(C) > M$ and that  $C'$ is a component of $F^{-m}(C)$. By Lemma~\ref{lem:mo}, the complexity does not increase under pullback, so we suppose for contradiction that  $N_{\mathcal{F}_0}(C') = N_{\mathcal{F}_0}(C)$. Let us derive a contradiction as follows.

Since \(F^{m}\) restricts to a homeomorphism on each \(L\in\mathcal F_0\), we have
\[
N(C,L)=\sum_{\gamma} N(\gamma,L), \qquad \forall L\in\mathcal F_0,
\]
where the sum ranges over all components \(\gamma\) of \(F^{-m}(C)\). Consequently, for any component \(C'\) of \(F^{-m}(C)\),
\[
N(C',L)\le N(C,L), \qquad \forall L\in\mathcal F_0,
\]
with equality holding if and only if the intersection number of \(L\) with every component of \(F^{-m}(C)\) other than \(C'\)  is zero.

It follows that
$$N_{\mathcal{F}_0}(C')=\sum_{L\in \mathcal{F}_0} N(C', L) \le \sum_{L\in \mathcal{F}_0}N(C, L)=N_{\mathcal{F}_0}(C).$$
By the  assumption, we must have $$N(C', L) = N(C, L), \:\:\forall L \in \mathcal{F}_0.$$
We claim that,  whenever  $C'$ enters the strip (the thin region) between  $L$ and its adjoint  arc $L'$ by crossing $L$ at a point $a$, it must eventually leave the strip by crossing $L'$ at some point $b$. Let us prove the claim by contradiction.  If not,  then since the strip contains no postcritical points, one could deform  $C'$ within the strip to reduce its intersection number with $L$.  As  the strip  contains no other arc in $\mathcal{F}_0$,   this  deformation  does not increase the intersection numbers with the  other  arcs.   This would imply that $N_{\mathcal{F}_0}(C') < N_{\mathcal{F}_0}(C)$, which contradicts the assumption that $N_{\mathcal{F}_0}(C') = N_{\mathcal{F}_0}(C)$. This proves the claim.

By the claim, it follows that
\[
N(C',L) \le N(C',L'),  \:\:\forall L \in \mathcal{F}_0.
\]
Since $F^m:L'\to L$ is a homeomorphism, we  have $$N(C',L')\le N(C,L),  \:\:\forall L \in \mathcal{F}_0.$$
Therefore,
\[
N(C',L') \le  N(C, L) = N(C', L) \le N(C', L'),  \:\:\forall L \in \mathcal{F}_0.
\]
Thus $$N(C',L') = N(C',L),  \:\:\forall L \in \mathcal{F}_0.$$

\begin{figure}
    \centering
    \includegraphics[width=0.65\linewidth]{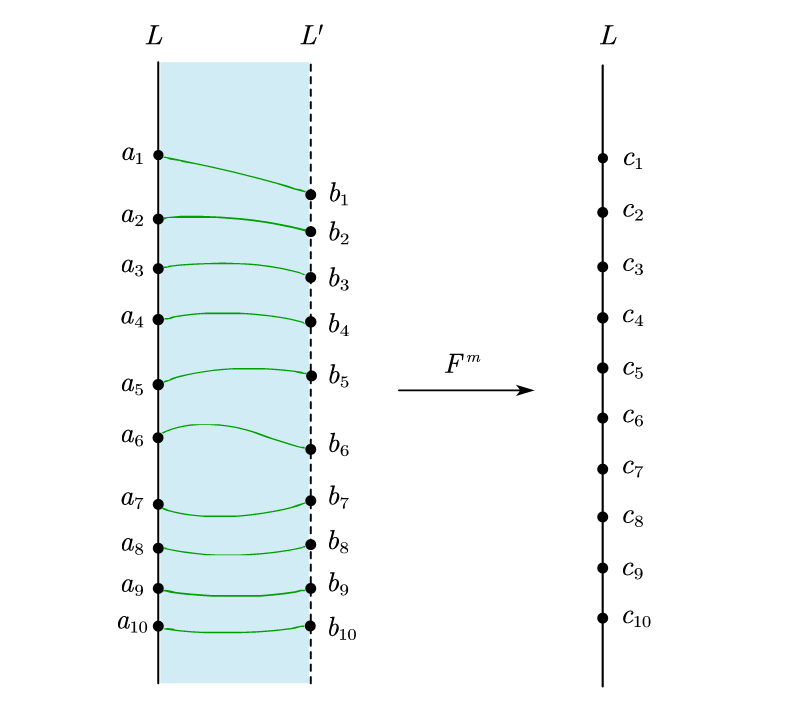}
 \caption{The points in the sets $C'\cap L$ and $C'\cap L'$ are mapped, in order, to the points in $C\cap L$.} 
    \label{F11}
\end{figure}

 Note that the points in $C' \cap L$, $C' \cap L'$ and $C \cap L$ are naturally ordered by the orientation of $L$ and $L'$, respectively. See Figure 11 for an illustration. For any two adjacent intersection points $a \in C' \cap L$ and $b \in C' \cap L'$ that are connected by an arc of $C'$ lying inside the strip,  the equality $N(C',L') = N(C',L)$ implies that  $a$ and $b$ occur in  the same order in $C' \cap L$ and $C' \cap L'$, respectively.    That is,  $a_i$ and $b_i$ are mapped by $F^m$ to $c_i$, with $$1 \le i \le    N(C',L') = N(C',L) = N(C, L).$$     So the arc of $C'$, joining $a_i$ and $b_i$ and lying in the strip bounded by $L$ and $L'$,  is mapped by $F^m$ onto the entire curve $C$.  Since for each $L \in \mathcal{F}_0$, there are exactly $N(C',L)$ such disjoint arcs, and since all such strips are disjoint with each other,
 the degree of  $F^m: C' \to C$  is  at least $$\sum_{L \in \mathcal{F}_0} N(C', L) = \sum_{L \in \mathcal{F}_0} N(C,  L)= \mathcal{N}_{\mathcal{F}_0}(C) >M.$$

This contradicts the fact that the topological degree of $F^m$ is $d^{rp}$. Hence the lemma holds.
\end{proof}

 By Lemmas~\ref{lem:mo} and ~\ref{prop:decay}, we have

\begin{cor}
\label{cor:bounded}
Let  $M = d^{rp}$.  Then for every  non-peripheral curve \(C \subset \widehat{\mathbb{C}} \setminus P_F\), there exists an integer $K \ge 1$, such that for all $k \ge K$,   every non-peripheral component \(C'\) of \(F^{-k}(C)\) satisfies $$N_{\mathcal{F}_0}(C') \le M.$$
\end{cor}

\begin{proof}[Proof of the Main Theorem]
For any non-peripheral curve \(C \subset \widehat{\mathbb{C}}\setminus P_F\), by Corollary~1, there exists a constant \(K \ge 0\) such that for every \(k \ge K\),  every non-peripheral component \(C'\) of \(F^{-k}(C)\) satisfies
\[
N_{\mathcal{F}_0}(C')\le M.
\]
 Let $s$ be the integer of Lemma~3.1. Then for any $k\ge K+s$ and any non-peripheral component $\eta$ of $F^{-k}(C)$,
\[
N_{\mathcal{F}}(\eta)\le c\,N_{\mathcal{F}_0}(F^{s}(\eta))\le cM.
\]
According to Section 3.1,   the curves with bounded complexity with respect to $\mathcal{F}$ must belong to a finite set of homotopy classes. Thus for all sufficiently large $k$, every non-peripheral component of $F^{-k}(C)$ lies in a fixed finite set of homotopy classes. The main theorem follows.
\end{proof}

The preceding considerations yield the following  contraction estimate.

\begin{cor}
There exist constants \(\alpha > 1\) and \(\beta, D > 0\), depending only on \(F\) and \(\mathcal{F}\), such that for every simple closed curve \(C \subset \widehat{\mathbb C} \setminus P_F\) and every \(n \ge 1\), the following estimate holds:
\[
\mathcal{N}_{\mathcal{F}}(\eta) \le
\begin{cases}
\alpha \mathcal{N}_{\mathcal{F}}(C) - \beta n + D, & \text{if } 1 \le n < \dfrac{\alpha}{\beta} \mathcal{N}_{\mathcal{F}}(C), \\[8pt]
D, & \text{if } n \ge \dfrac{\alpha}{\beta} \mathcal{N}_{\mathcal{F}}(C),
\end{cases}
\]
where \(\eta\) is any component of \(F^{-n}(C)\).
\end{cor}

\end{document}